\documentclass[11pt,leqno]{article}
\usepackage[doc]{optional}
\usepackage{color}
\usepackage{float}
\usepackage{soul}
\usepackage{graphicx}
\definecolor{labelkey}{rgb}{0,0.08,0.45}
\definecolor{refkey}{rgb}{0,0.6,0.0}
\definecolor{Brown}{rgb}{0.45,0.0,0.05}
\definecolor{lime}{rgb}{0.00,0.8,0.0}
\definecolor{lblue}{rgb}{0.5,0.5,0.99}

\usepackage{mathptmx}

\usepackage{amsmath}

\usepackage{stmaryrd}
\usepackage{amssymb}
\usepackage{theorem}
\newcommand{\nnn}{\ensuremath{{n\in{\mathbb N}}}}
\newcommand{\thalb}{\ensuremath{\tfrac{1}{2}}}
\newcommand{\menge}[2]{\big\{{#1}~\big |~{#2}\big\}}

\newcommand{\To}{\ensuremath{\rightrightarrows}}

\newcommand{\fenv}[1]%
{\ensuremath{\,\overrightarrow{\operatorname{env}}_{#1}}}
\newcommand{\benv}[1]%
{\ensuremath{\,\overleftarrow{\operatorname{env}}_{#1}}}
\newcommand{\emp}{\ensuremath{\varnothing}}

\newcommand{\scal}[2]{\left\langle{#1},{#2}  \right\rangle}

\newcommand{\exi}{\ensuremath{\exists\,}}
\newcommand{\zeroun}{\ensuremath{\left]0,1\right[}}
\newcommand{\RR}{\ensuremath{\mathbb R}}
\newcommand{\RP}{\ensuremath{\mathbb{R}_+}}
\newcommand{\RM}{\ensuremath{\mathbb{R}_-}}

\newcommand{\KK}{\ensuremath{\mathbf K}}
\newcommand{\ZZ}{\ensuremath{\mathbf Z}}

\newcommand{\dom}{\ensuremath{\operatorname{dom}}}

\newcommand{\gr}{\ensuremath{\operatorname{gr}}}

\newcommand{\reli}{\ensuremath{\operatorname{ri}}}
\newcommand{\inte}{\ensuremath{\operatorname{int}}}
\newcommand{\sri}{\ensuremath{\operatorname{sri}}}

\newcommand{\ran}{\ensuremath{\operatorname{ran}}}
\newcommand{\zer}{\ensuremath{\operatorname{zer}}}

\newcommand{\cspan}{\ensuremath{\overline{\operatorname{span}}\,}}

\newcommand{\Fix}{\ensuremath{\operatorname{Fix}}}
\newcommand{\Id}{\ensuremath{\operatorname{Id}}}

\newcommand{\minf}{\ensuremath{-\infty}}

\newcommand{\minimize}[2]{\ensuremath{\underset{\substack{{#1}}}{\mathrm{minimize}}\;\;#2 }}
\newcommand{\veet}{\ensuremath{{\scriptscriptstyle\vee}}} % vee tiny

{\begin{list}{}{%
\settowidth{\labelwidth}{\textrm{#1~}}%
\setlength{\leftmargin}{\labelwidth+\labelsep}}}%requires macro calc.sty
{\end{list}}
\newtheorem{theorem}{Theorem}[section]

\newtheorem{corollary}[theorem]{Corollary}
\newtheorem{proposition}[theorem]{Proposition}
\newtheorem{definition}[theorem]{Definition}
\theoremstyle{plain}{\theorembodyfont{\rmfamily}
}
\theoremstyle{plain}{\theorembodyfont{\rmfamily}
}
\theoremstyle{plain}{\theorembodyfont{\rmfamily}
}
\theoremstyle{plain}{\theorembodyfont{\rmfamily}
\newtheorem{example}[theorem]{Example}}
\newtheorem{fact}[theorem]{Fact}
\theoremstyle{plain}{\theorembodyfont{\rmfamily}
\newtheorem{remark}[theorem]{Remark}}

\newcommand{\boxedeqn}[1]{%
    \[\fbox{%
        \addtolength{\linewidth}{-2\fboxsep}%
        \addtolength{\linewidth}{-2\fboxrule}%
        \begin{minipage}{\linewidth}%
        \begin{equation}#1\\[+5mm]\end{equation}%
        \end{minipage}%
      }\]%
  }

\allowdisplaybreaks % or locally if problems {\allowdisplaybreaks
\begin{document}

\title{\textsc{
Attouch-Th\'era duality revisited:\\
paramonotonicity and operator splitting}}

\author{
Heinz H.\ Bauschke\thanks{
Mathematics, University
of British Columbia,
Kelowna, B.C.\ V1V~1V7, Canada. E-mail:
\texttt{heinz.bauschke@ubc.ca}.},
Radu I.\ Bo\c{t}\thanks{
Technische Universit\"{a}t Chemnitz,
Fakult\"{a}t f\"{u}r Mathematik,
09107 Chemnitz, Germany. E-mail:
\texttt{radu.bot@mathematik.tu-chemnitz.de}.},
Warren L.\ Hare\thanks{
Mathematics,
University of British Columbia,
Kelowna, B.C.\ V1V~1V7, Canada.
E-mail:  \texttt{warren.hare@ubc.ca}.},
~and Walaa M.\ Moursi\thanks{
Mathematics, University of
British Columbia,
Kelowna, B.C.\ V1V~1V7, Canada. E-mail:
\texttt{walaa.moursi@ubc.ca}.}}

\date{October 21, 2011}
\maketitle

\vskip 8mm

\begin{abstract} \noindent
The problem of finding the zeros of the sum of two
maximally monotone operators is of fundamental importance
in optimization and variational analysis.
In this paper, we systematically study Attouch-Th\'era duality for this
problem. We provide new results
related to Passty's parallel sum, to Eckstein and Svaiter's extended
solution set, and to Combettes' fixed point description of the
set of primal solutions.
Furthermore,
paramonotonicity is revealed to be a key property 
because it allows for the recovery
of \emph{all} primal solutions given just 
\emph{one arbitrary} dual solution. 
As an application, we generalize the best approximation results
by Bauschke, Combettes and Luke
[J.~Approx.~Theory~141 (2006), 63--69] from normal cone operators 
to paramonotone operators. 
Our results are illustrated through numerous examples.

\end{abstract}

{\small

}

\section{Introduction}

Throughout this paper,
\boxedeqn{
\text{$X$ is
a real Hilbert space with inner
product $\scal{\cdot}{\cdot}$ }
}
and induced norm $\|\cdot\|$.

Let $A\colon X\To X$ be a set-valued operator, i.e.,
$(\forall x\in X)$ $Ax\subseteq X$. 
Recall that $A$ is \emph{monotone} if
\begin{equation}
(\forall (x,x^*)\in\gr A)(\forall (y,y^*)\in\gr A)
\quad
\scal{x-y}{x^*-y^*} \geq 0
\end{equation}
and that $A$ is \emph{maximally monotone} if it is impossible
to properly enlarge the graph of $A$ while keeping monotonicity.
Monotone operators continue to play an important role
in modern optimization and variational analysis;
see, e.g., 
\cite{BC2011},
\cite{BorVanBook}, 
\cite{Brezis},
\cite{BurIus},
\cite{Rock70},
\cite{Rock98},
\cite{Simons1},
\cite{Simons2},
\cite{Zalinescu},
\cite{Zeidler2a},
\cite{Zeidler2b},
and \cite{Zeidler1}.
This is due to the fact that subdifferential operators of 
proper lower semicontinuous convex functions
are maximally monotone, as are continuous linear operators with a
monotone symmetric part. 
The sum of two maximally monotone operators is monotone, and often
maximally monotone if an appropriate constraint qualification is imposed.
Finding the zeros of two maximally monotone operators $A$ and $B$,
i.e., determining 
\begin{equation}
(A+B)^{-1}0 = \menge{x\in X}{0 \in Ax+Bx}, 
\end{equation}
is a problem of great interest
because it covers constrained convex optimization, convex feasibility, and
many others. Attouch and Th\'era provided a comprehensive study of this
(primal) problem in terms of \emph{duality}.
Specifically, they associated with the primal problem
a dual problem.
We set $B^{-\ovee} = (-\Id)\circ B^{-1} \circ (-\Id)$
where $\Id\colon X\to X\colon x\mapsto x$ is the identity operator.
The \emph{Attouch-Th\'era dual problem} is then to determine
\begin{equation}
\big(A^{-1}+B^{-\ovee}\big)^{-1}0 = \menge{x^*\in X}{0 \in
A^{-1}x^*+B^{-\ovee}x^*}. 
\end{equation}
This duality is very beautiful; e.g.,
the dual of the dual problem is the primal problem,
and the primal problem possesses at least one solution
if and only if the same is true for the dual problem. 

\emph{
Our goal in this paper is to systematically study 
Attouch-Th\'era duality, to derive new results, 
and to expose new applications.
}

Let us now summarize our main results.
\begin{itemize}
\item We observe a curious convexity property of the intersection of two
sets involving the graphs of $A$ and $B$ (see Theorem~\ref{t:friday}). 
This relates to Passty's work on the parallel sum as well as to
Eckstein and Svaiter's work on the extended solution set.
\item We provide a new description of the fixed point set of the
Douglas-Rachford splitting operator (see Theorem~\ref{t:mace2});
this refines Combettes' description of $(A+B)^{-1}0$. 
\item We reveal the importance of paramonotonicity:
in this case, the fixed point set of the Douglas-Rachford splitting
operator is a \emph{rectangle} (see Corollary~\ref{c:2GB}) 
and it is possible to recover
\emph{all} primal solutions from \emph{one} dual solution 
(see Theorem~\ref{t:apes}). 
\item We generalize the best approximation results by
Bauschke-Combettes-Luke from normal cone operators to 
paramonotone operators with a common zero
(see Corollary~\ref{c:lousypizza} and Theorem~\ref{t:abstract}). 
\end{itemize}

The remainder of this paper is organized as follows.
In Section~\ref{s:basicduality}, 
we review and slightly refine the basic results on
Attouch-Th\'era duality.
The solution mappings between primal and dual solutions
are studied in Section~\ref{s:solmaps}. 
Section~\ref{s:splitop} deals with the Douglas-Rachford splitting
operator. The results in Section~\ref{s:para} and
Section~\ref{s:poss} underline the importance
of paramonotonicity in the understanding of the zeros of the sum.
Applications to best approximation as well as comments 
on other duality framework are the topic of the final
Section~\ref{s:final}. 

We conclude this introductory section with some notational comments.
The set of zeros of $A$ is written as $\zer A = A^{-1}0$.
The \emph{resolvent} and \emph{reflected resolvent} is defined by
\begin{equation}
\label{e:alden}
J_A = (\Id + A)^{-1}
\;\text{and}\;
R_A = 2J_A-\Id,
\end{equation}
respectively.
It is well known that $\zer A = \Fix J_A := \menge{x\in X}{J_Ax=x}$. 
Moreover,
$J_A$ is firmly nonexpansive if and only if $R_A$ is nonexpansive;
see, e.g.,  \cite{EckBer}, \cite{LawSpin}, or \cite{Minty}. 
We also have the \emph{inverse resolvent identity}
\begin{equation}
\label{e:iri}
J_A + J_{A^{-1}} = \Id 
\end{equation}
and the following very useful Minty parametrization.

\begin{fact}[Minty parametrization]
\label{f:Mintypar}
Let $A\colon X\To X$ be maximally monotone.
Then
$\gr A\to X\colon (a,a^*)\mapsto a+a^*$
is a continuous bijection with continuous
inverse $x\mapsto (J_Ax,x-J_Ax)$; thus,
\begin{equation}
\gr A = \menge{(J_Ax,x-J_Ax)}{x\in X}.
\end{equation}
\end{fact}

Without explicit mentioning it, we employ
standard notation from convex analysis
(see \cite{Rock70}, \cite{Rock98}, or \cite{Zalinescu}). 
Most importantly,
$f^*$ denotes the \emph{Fenchel conjugate} of a function $f$,
and $\partial f$ its \emph{subdifferential operator}. 
The set of all convex lower semicontinuous proper functions on $X$ is
denoted by $\Gamma$ (or $\Gamma_X$ if we need to emphasize the space). 
Finally, we set 
$f^\veet := f\circ (-\Id)$, which yields $\partial(f^\veet) = (\partial
f)^\ovee$.

\section{Duality for monotone operators}

\label{s:basicduality}

In this paper, we study the problem of finding zeros of the sum
of maximally monotone operators. More specifically, we
assume that
\boxedeqn{
\text{$A$ and $B$ are maximally monotone operators on $X$.}
}

\begin{definition}[primal problem]
The \emph{primal problem}, for the ordered pair $(A,B)$,
is to find the zeros of $A+B$.
\end{definition}

At first, it looks strange to define the primal problem with respect to the
(ordered) pair $(A,B)$. The reason we must do this is to associate 
a \emph{unique} dual problem. (The ambiguity arises because addition
is commutative.) It will be quite convenient to set
\boxedeqn{
%\moyo{A}{\vee}^{\vee} = (-\Id)\circ A \circ(-\Id).
A^\ovee = (-\Id)\circ A \circ(-\Id).
}
An easy calculation shows that
$ (A^{-1})^\ovee = (A^{\ovee})^{-1}$, which motivates the notation
\boxedeqn{
\label{e:commute}
A^{-\ovee} := \big(A^{-1}\big)^\ovee = \big(A^{\ovee}\big)^{-1}.
}
(This is similar to the linear-algebraic
notation $A^{-T}$ for invertible square matrices.)

Now since $A$ and $B$ form a pair of maximally monotone operators,
so do $A^{-1}$ and $B^{-\ovee}$:
We thus define the \emph{dual pair}
\begin{equation}
(A,B)^* := (A^{-1},B^{-\ovee}).
\end{equation}
The biduality
\begin{equation}
\label{e:0823:a}
(A,B)^{**} = (A,B)
\end{equation}
holds, since 
$(A^{-1})^{-1}=A$, $(B^{\ovee})^\ovee = B$, and 
$(B^\ovee)^{-1}=(B^{-1})^\ovee$.

We are now in a position to formulate the dual problem.

\begin{definition}[(Attouch-Th\'era) dual problem]
The \emph{(Attouch-Th\'era) dual problem}, for the ordered pair $(A,B)$,
is to find the zeros of $A^{-1} + B^{-\ovee}$.
Put differently: the dual problem for $(A,B)$
is precisely the primal problem for $(A,B)^*$.
\end{definition}

In view of \eqref{e:0823:a}, it is the clear that the primal
problem is precisely the dual of the dual problem, as expected.
One central aim of this paper is to understand the interplay
between the primal and dual solutions that we formally define next.

\begin{definition}[primal and dual solutions]
\label{d:pdsols}
The \emph{primal solutions} are the solutions to the primal problem
and analogously for the \emph{dual solutions}.
We shall abbreviate these sets by
\boxedeqn{
Z := (A+B)^{-1}(0)
\quad\text{and}\quad
K := \big( A^{-1} + B^{-\ovee}\big)^{-1}(0),
}
respectively.
\end{definition}

As observed by Attouch and Th\'era in \cite[Corollary~3.2]{AT},
one has:
\begin{equation}
Z\neq\emp \quad \Leftrightarrow \quad  K\neq\emp.
\end{equation}
Let us make this simple but important equivalence
a little more precise. In order to do so, we
define
\boxedeqn{
\label{e:K_z}
(\forall z\in X)\quad
K_z := (Az) \cap (-Bz)
}
and
\boxedeqn{
\label{e:Z_k}
(\forall k\in X)\quad
Z_k := (A^{-1}k) \cap (-B^{-\ovee}k)
= (A^{-1}k) \cap \big(B^{-1}(-k)\big).
}

As the next proposition illustrates, these objects are intimately
tied to primal and dual solutions defined in
Definition~\ref{d:pdsols}.

\begin{proposition}
\label{p:momday}
Let $z\in X$ and let $k\in X$.
Then the following hold.
\begin{enumerate}
\item
\label{p:momdayi--}
$K_z$ and $Z_k$ are closed convex (possibly empty) subsets of $X$.

\item
\label{p:momdayi-}
$k\in K_z$ $\Leftrightarrow$
$z\in Z_k$;
\item
\label{p:momdayi}
$z\in Z$ $\Leftrightarrow$ $K_z\neq\emp$.
\item
\label{p:momdayii}
$\bigcup_{z\in Z} K_z = K$.
\item
\label{p:momdayiii}
$Z\neq\emp$ $\Leftrightarrow$ $K\neq\emp$.
\item
\label{p:momdayiv}
$k\in K$ $\Leftrightarrow$ $Z_k\neq\emp$.
\item
\label{p:momdayv}
$\bigcup_{k\in K} Z_k = Z$.
\end{enumerate}
\end{proposition}
\begin{proof}
\ref{p:momdayi--}:
Because $A$ and $B$ are maximally monotone,
the sets $Az$ and $Bz$ are closed and convex.
Hence $K_z$ is also closed and convex.
We see analogously that $Z_k$ is closed and convex as well.

\ref{p:momdayi-}:
This is easily verified from the definitions.

\ref{p:momdayi}:
Indeed,
$z\in Z$
$\Leftrightarrow$
$0\in (A+B)z$
$\Leftrightarrow$
$(\exi a^*\in Az\cap (-Bz))$
$\Leftrightarrow$
$(\exi a^*\in K_z)$
$\Leftrightarrow$
$K_z\neq\emp$.

\ref{p:momdayii}:
Take $k\in\bigcup_{z\in Z}K_z$.
Then there exists $z\in Z$ such that
$k\in K_z = Az\cap (-Bz)$.
Hence $z\in A^{-1}k$
and $z\in (-B)^{-1}k = B^{-1}(-\Id)^{-1}k = B^{-1}(-k)$.
Thus $z\in A^{-1}k$ and $-z\in B^{-\ovee}k$.
Hence $0 \in (A^{-1}+B^{-\ovee})k$ and so $k\in K$.
The reverse inclusion is proved analogously.

\ref{p:momdayiii}:
Combine \ref{p:momdayi} and \ref{p:momdayii}.

\ref{p:momdayiv}\&\ref{p:momdayv}:
The proofs are analogous to the ones of \ref{p:momdayi}\&\ref{p:momdayii}.
\end{proof}

Let us provide some examples illustrating these notions.

\begin{example}
\label{ex:skewskew}
Suppose that $X=\RR^2$, and that
 we consider the rotators by $\mp\pi/2$, i.e.,
\begin{equation}
A\colon \RR^2\to\RR^2\colon (x_1,x_2)\mapsto (x_2,-x_1)
\quad\text{and}\quad
B\colon \RR^2\to\RR^2\colon (x_1,x_2)\mapsto (-x_2,x_1).
\end{equation}
Note that  $B=-A = A^{-1} = A^*$, where $A^*$ denote the adjoint
operator.
Hence $A+B\equiv 0$, $Z=X$, and $(\forall z\in Z)$
$K_z = \{Az\} = \{-Bz\}$.
Furthermore,
$A^{-1} = B$ while
the linearity of $B$ implies that $B^{-\ovee} = B^{-1} = -B = A$.
Therefore, $(A,B)^* = (B,A)$. Hence
$K=Z$, while $(\forall k\in K)$ $Z_k = \{A^{-1}k\} = \{Bk\}$.
\end{example}

\begin{example}
\label{ex:normskew}
Suppose that $X=\RR^2$,
that $A$ is the normal cone operator of $\RR^2_+$,
and that $B\colon X\to X\colon (x_1,x_2)\mapsto (-x_2,x_1)$
is the rotator by $\pi/2$.
As already observed in Example~\ref{ex:skewskew}, we have
$B^{-1}=-B$ and $B^{-\ovee} = B^{-1} = -B$.
A routine calculation yields
\begin{equation}
Z = \RP\times\{0\};
\end{equation}
thus, since $B$ is single-valued,
\begin{equation}
(\forall z = (z_1,0)\in Z)\quad
K_z = \big\{-Bz\big\}
= \big\{(0,-z_1)\big\}.
\end{equation}
Thus,
\begin{equation}
K = \bigcup_{z\in Z} K_z = \{0\}\times\RM
\end{equation}
and so
\begin{equation}
(\forall k = (0,k_2)\in K)\quad
Z_k = \big\{-B^{-\ovee}k\big\}
= \big\{Bk\big\}
= \big\{(-k_2,0)\big\}.
\end{equation}
The dual problem is to find the zeros of $A^{-1} +
B^{-\ovee}$, i.e., the zeros of the sum
of the normal cone operator of the negative orthant and the rotator by
$-\pi/2$.
\end{example}

\begin{example}[convex feasibility]
\label{ex:cf}
Suppose that $A=N_U$ and $B=N_V$, where
$U$ and $V$ are closed convex subsets of $X$
such that $U\cap V\neq\varnothing$.
Then clearly $Z = U \cap V$.
Using
\cite[Proposition~2.4.(i)]{BCL04}, we deduce that
$(\forall z\in Z)$ $K_z = N_{\overline{U-V}}(0) = K$.
Note that we do know at least one dual solution: 
$0\in K$.
Thus, by Proposition~\ref{p:momday}\ref{p:momdayi-}\&\ref{p:momdayv},
$(\forall k\in K)$ $Z_k = Z$.
\end{example}

\begin{remark}
The preceding examples
give some credence to the conjecture that
\begin{equation}
\label{e:nopartition}
\left.
\begin{matrix}
z_1\in Z\\
z_2\in Z\\
z_1\neq z_2
\end{matrix}
\right\}
\quad\Rightarrow\quad
\text{either $K_{z_1}=K_{z_2}$ or $K_{z_1}\cap K_{z_2}=\varnothing$.}
\end{equation}
Note that \eqref{e:nopartition} is trivially true whenever
$A$ or $B$ is at most single-valued.
While this conjecture fails in general (see Example~\ref{ex:WW} below),
it does, however, hold true for
the large class of paramonotone operators (see Theorem~\ref{t:apes}).
\end{remark}

\begin{example}
\label{ex:WW}
Suppose that $X=\RR^2$, and
set $U := \RR\times\RR_+$, $V = \RR\times\{0\}$,
and $R\colon X\to X\colon (x_1,x_2)\mapsto (-x_2,x_1)$.
Now suppose that
$A = N_U + R$ and that $B=N_V$.
Then $\dom A = U$ and $\dom B=V$;
hence, $\dom(A+B) = U\cap V = V$.
Let $x = (\xi,0)\in V$.
Then $Ax = \{0\}\times\left]\minf,\xi\right]$
and $Bx = \{0\}\times\RR$.
Hence $Ax\subset \pm Bx$, $(A+B)x=\{0\}\times\RR$ and therefore $Z=V$.
Furthermore,
$K_x = Ax\cap (-Bx) = Ax$.
Now take $y = (\eta,0)\in V = Z$ with $\xi<\eta$.
Then $K_x = Ax\subsetneqq Ay = K_y$ and thus \eqref{e:nopartition}
fails.
\end{example}

\begin{proposition}[common zeros]
\label{p:commonzeros}
$\zer A\cap \zer B\neq\varnothing$
$\Leftrightarrow$
$0\in K$.
\end{proposition}
\begin{proof}
Suppose first that $z\in\zer A\cap\zer B$.
Then $0\in Az$ and $0\in Bz$,
so $0 \in Az\cap (-Bz) = K_z\subseteq K$.
Now assume that $0\in K$.
Then $0\in K_z$, for some $z\in Z$ and
so $0\in Az\cap(-Bz)$. 
Therefore, $0\in\zer A\cap \zer B$.
\end{proof}

\begin{example}
Suppose that $B=A$.
Then $Z = \zer A$, and
$\zer A\neq\varnothing$
$\Leftrightarrow$
$0\in K$.
\end{example}
\begin{proof}
Since $2A$ is maximally monotone and
$A+A$ is a monotone extension of $2A$, we deduce that
that $A+A=2A$.
Hence $Z= \zer(2A) = \zer A$ and the result follows
from Proposition~\ref{p:commonzeros}.
\end{proof}

The following result, observed first by Passty, is very useful.
For the sake of completeness, we include its short proof.

\begin{proposition}[Passty]
\label{p:Passty}
Suppose that, for every $i\in\{0,1\}$, 
$w_i\in Ay_i \cap B(x-y_i)$.
Then
$\scal{y_0-y_1}{w_0-w_1}=0$.
\end{proposition}
\begin{proof}
(See \cite[Lemma~14]{Passty86}.)
Since $A$ is monotone,
$0\leq \scal{y_0-y_1}{w_0-w_1}$.
On the other hand, since $B$ is monotone,
$0\leq\scal{(x-y_0)-(x-y_1)}{w_0-w_1} = \scal{y_1-y_0}{w_0-w_1}$.
Altogether, 
$\scal{y_0-y_1}{w_0-w_1} = 0$. 
\end{proof}

\begin{corollary}
\label{c:Passty}
Suppose that $z_1$ and $z_2$ belong to $Z$, that
$k_1\in K_{z_1}$, and that $k_2\in K_{z_2}$.
Then $\scal{k_1-k_2}{z_1-z_2}=0$.
\end{corollary}
\begin{proof}
Apply Proposition~\ref{p:Passty}
(with $B$ replaced by $B^\ovee$ and at $x=0$). 
%$(\forall i\in \{1,2\})$
%$k_i \in K_{z_i}$
%$\Leftrightarrow$ $k_i \in (Az_i) \cap (-Bz_i)$.
%Thus, the monotonicity of $B$ yields
%\begin{equation}
%0 \leq \scal{(-k_1)-(-k_2)}{z_1-z_2} = \scal{k_2-k_1}{z_1-z_2},
%\end{equation}
%while the monotonicity of $A$ implies
%\begin{equation}
%0 \leq \scal{k_1-k_2}{z_1-z_2}.
%\end{equation}
%Altogether, $\scal{k_1-k_2}{z_1-z_2}=0$.
\end{proof}

\section{Solution mappings $\KK$ and $\ZZ$}

\label{s:solmaps}

We now interpret the families of sets $(K_z)_{z\in X}$
and $(Z_k)_{k\in X}$ as set-valued operators by setting
\boxedeqn{
\KK\colon X\To X\colon z\mapsto K_z
\quad\text{and}\quad
\ZZ\colon X\To X\colon k\mapsto Z_k.
}

Let us record some basic properties of these
fundamental operators. 

\begin{proposition}
\label{p:ZZKK}
The following hold.
\begin{enumerate}
\item
\label{p:ZZKKi}
$\gr\KK=\gr A \cap \gr(-B)$
and $\gr\ZZ = \gr A^{-1}\cap \gr (-B^{-\ovee})$.
\item
\label{p:ZZKKi+}
$\dom\KK = Z$, $\ran\KK=K$, $\dom \ZZ=K$, and $\ran\ZZ = Z$.
\item
\label{p:ZZKKii}
$\gr\KK$ and $\gr\ZZ$ are closed sets.
\item
\label{p:ZZKKiii}
The operators $\KK,-\KK,\ZZ,-\ZZ$ are monotone.
\item
\label{p:ZZKKiv}
$\KK^{-1} = \ZZ$.

\end{enumerate}
\end{proposition}
\begin{proof}
\ref{p:ZZKKi}:
This is clear from the definitions.

\ref{p:ZZKKi+}:
This follows from Proposition~\ref{p:momday}.

\ref{p:ZZKKii}:
Since $A$ and $B$ are maximally monotone,
the sets $\gr A$ and $\gr B$ are closed.
Hence, by \ref{p:ZZKKi}, $\gr\KK$ is closed and similarly
for $\gr \ZZ$.

\ref{p:ZZKKiii}:
Since $\gr\KK \subseteq\gr A$ and $A$ is monotone, we see that
$\KK$ is monotone.
Similarly, since $B$ is monotone and $\gr(-\KK)\subseteq\gr B$,
we obtain the monotonicity of $-\KK$.
The proofs for $\pm\ZZ$ are analogous.

\ref{p:ZZKKiv}:
Clear from Proposition~\ref{p:momday}\ref{p:momdayi-}.
\end{proof}

In Proposition~\ref{p:momday}\ref{p:momdayi}
we observed the closedness and convexity
of $K_z$ and $Z_k$.
In view of Proposition~\ref{p:momday}\ref{p:momdayi}\&\ref{p:momdayv},
the sets of primal and dual solutions are both
unions of closed convex sets.
It would seem that we cannot a priori deduce convexity of these
solution sets because unions of convex sets need not be convex.
However, not only are $Z$ and $K$ indeed convex, but so are
$\gr\ZZ$ and $\gr\KK$.
This surprising result, which is basically contained in works
by Passty \cite{Passty86} and by Eckstein and Svaiter
\cite{EckSvai08,EckSvai09}, is best stated by using 
the parallel sum, a notion systematically explored
by Passty in  \cite{Passty86}.

\begin{definition}[parallel sum]
The \emph{parallel sum} of $A$ and $B$ is 
\begin{equation}
A \Box B := (A^{-1}+B^{-1})^{-1}.
\end{equation}
\end{definition}
The notation we use for the parallel sum 
(see \cite[Section~24.4]{BC2011}) is
nonstandard but highly convenient: indeed,
for sufficiently nice convex functions
$f$ and $g$, one has
$\partial(f\Box g) = (\partial f)\Box (\partial g)$
(see \cite[Theorem~28]{Passty86}, \cite[Proposition~4.2.2]{Moudafi96},
or \cite[Proposition~24.27]{BC2011}).

The proof of the following result is contained in the proof of
\cite[Theorem~21]{Passty86}, although Passty stated a much weaker
conclusion. For the sake of completeness, we present his proof.

\begin{theorem}
\label{t:friday}
For every $x\in X$, the set 
\begin{equation}
\label{e:friday1}
\big(\gr A \big)
\cap
\big((x,0)-\gr(-B)\big)
=\menge{(y,w)\in\gr A}{(x-y,w)\in\gr B}
\end{equation}
is convex. 
\end{theorem}
\begin{proof}
(See also \cite[Proof of Theorem~21]{Passty86}.) 
The identity \eqref{e:friday1} is easily verified. 
To tackle convexity, 
for every $i\in\{0,1\}$ take 
$(y_i,w_i)$ from the intersection \eqref{e:friday1}; equivalently, 
\begin{equation}
\label{e:mace0}
(\forall i\in\{0,1\})\quad
w_i\in Ay_i \cap B(x-y_i).
\end{equation}
By Proposition~\ref{p:Passty}, 
\begin{equation}
\label{e:mace0+}
\scal{y_0-y_1}{w_0-w_1}=0.
\end{equation}
Now let $t\in[0,1]$,
set $(y_t,w_t) = (1-t)(y_0,w_0) + t(y_1,w_1)$,
and take $(a,a^*)\in\gr A$. 
Using \eqref{e:mace0} and the monotonicity of $A$ in \eqref{e:mace1}, 
we obtain
\begin{subequations}
\label{e:friday2}
\begin{align}
\scal{ y_t - a}{w_t-a^*}
&= \scal{(1-t)(y_0-a)+t(y_1-a)}{(1-t)(w_0-a^*) + t(w_1-a^*)}\\
&= (1-t)^2\scal{y_0-a}{w_0-a^*} + t^2\scal{y_1-a}{w_1-a^*}\\
&\quad + (1-t)t\big(\scal{y_0-a}{w_1-a^*}+\scal{y_1-a}{w_0-a^*} \big)\\
&\geq  (1-t)t\big(\scal{y_0-a}{w_1-a^*}+\scal{y_1-a}{w_0-a^*} \big).
\label{e:mace1}
\end{align}
\end{subequations}
Thus, using again monotonicity of $A$ and recalling
\eqref{e:mace0+}, we obtain 
\begin{subequations}
\label{e:friday3}
\begin{align}
\scal{y_0-a}{w_1-a^*}+\scal{y_1-a}{w_0-a^*}
&= \scal{y_0-a}{w_1-w_0} + \scal{y_0-a}{w_0-a^*}\\
&\quad + \scal{y_1-a}{w_0-w_1}+\scal{y_1-a}{w_1-a^*}\\
&= \scal{y_1-y_0}{w_0-w_1}\\
&\quad + \scal{y_0-a}{w_0-a^*} + \scal{y_1-a}{w_1-a^*}\\
&\geq \scal{y_1-y_0}{w_0-w_1}\\
&= 0.
\end{align}
\end{subequations}
Combining \eqref{e:friday2} and \eqref{e:friday3},
we obtain $\scal{y_t-a}{w_t-a^*}\geq 0$.
Since $(a,a^*)$ is an arbitrary element of $\gr A$ and $A$ is maximally
monotone, we deduce that
$(y_t,w_t)\in\gr A$.
A similar argument yields $(x-y_t,w_t)\in\gr B$. 
Therefore, $(y_t,w_t)$ is an element of the intersection \eqref{e:friday1}.
\end{proof}

Before returning to the objects of interest, we record Passty's 
\cite[Theorem~21]{Passty86} as a simple corollary.

\begin{corollary}[Passty]
\label{c:passty} 
For every $x\in X$, the set $(A\Box B)x$ is convex.
\end{corollary}
\begin{proof}
Let $x\in X$. 
Since $(y,w)\to w$ is linear
and $\menge{(y,w)\in\gr A}{(x-y,w)\in\gr B}$
is convex (Theorem~\ref{t:friday}), we deduce that
\begin{equation}
\menge{w\in X}{(\exi y\in X)\; w\in Ay\cap B(x-y)}
\text{ is convex.}
\end{equation}
On the other hand, 
a direct computation or \cite[Lemma~2]{Passty86} implies that
$(A\Box B)x = \bigcup_{y\in X}\big( Ay\cap B(x-y)\big)$.
Altogether, $(A\Box B)x$ is convex. 
\end{proof}

\begin{corollary}
\label{c:strange}
For every $x\in X$, the set $(\gr A)\cap ((x,0)+\gr(-B))$ is convex.
\end{corollary}
\begin{proof}
On the one hand, $-\gr(-B^\ovee)=\gr(-B)$. 
On the other hand, $B^\ovee$ is maximally monotone.
Altogether, Theorem~\ref{t:friday} (applied with $B^\ovee$ instead of $B$)
implies that 
$(\gr A) \cap ((x,0)-\gr(-B^\ovee)) = (\gr A) \cap ((x,0)+\gr(-B))$
is convex. 
\end{proof}

\begin{remark}
Theorem~\ref{t:friday} and Corollary~\ref{c:strange}
imply that the intersections $(\gr A)\cap \pm(\gr -B)$ are convex. 
This somewhat resembles works by Mart{\'{\i}}nez-Legaz 
(see \cite[Theorem~2.1]{ML08}) and by Z\u{a}linescu \cite{Zali06},
who encountered convexity when studying the Minkowski sum/difference 
$(\gr A)\pm (\gr -B)$. 
\end{remark}

\begin{corollary}[convexity]
\label{c:convex}
The sets $\gr\ZZ$ and $\gr\KK$ are convex;
consequently,
$Z$ and $K$ are convex.
\end{corollary}
\begin{proof}
Combining Proposition~\ref{p:ZZKK}\ref{p:ZZKKi} and 
Corollary~\ref{c:strange} (with $x=0$), we obtain
the convexity of $\gr\KK$.
Hence $\gr\ZZ$ is convex by Proposition~\ref{p:ZZKK}\ref{p:ZZKKiv}. 
It thus follows that $Z$ and $K$
are convex
as images of convex sets under linear transformations.
\end{proof}

\begin{remark}
\label{r:parallel}
Since 
$Z = (A^{-1}\Box B^{-1})(0)$ and
$K = (A\Box B^\ovee)(0)$,
the convexity of $Z$ and $K$ also follows from 
Corollary~\ref{c:passty}. 
\end{remark}

\begin{remark}[connection to Eckstein and Svaiter's ``extended
solution set'']
In \cite[Section~2.1]{EckSvai08}, Eckstein and Svaiter defined in 2008
the \emph{extended solution set} (for the primal problem) by
\begin{equation}
S_e(A,B) := \menge{(z,w)\in X\times X}{w\in Bz,\,-w\in Az}.
\end{equation}
It is clear that $\gr\ZZ^{-1}=\gr\KK = S_e(B,A) = -S_e(A,B)$.
Unaware of Passty's work,
they proved in \cite[Lemma~1 and Lemma~2]{EckSvai08}
(in the present notation) that $Z = \ran\ZZ$, and that $\gr\ZZ$
is closed and convex. Their proof is very elegant and completely different
from the above Passty-like proof.
In their 2009 follow-up paper \cite{EckSvai09},
Eckstein and Svaiter generalize the notion of the
extended solution set to three or more operators; their
corresponding proof of convexity in
\cite[Proposition~2.2]{EckSvai09} is more direct and along Passty's lines.
\end{remark}

\begin{remark}[convexity of $Z$ and $K$]
If $Z$ is nonempty and a constraint qualification holds,
then $A+B$ is maximally monotone (see, e.g.,
\cite[Section~24.1]{BC2011}) and therefore
$Z = \zer(A+B)$ is convex.
It is somewhat surprising that $Z$ is always convex even without the maximal
monotonicity of $A+B$.
\end{remark}

One may inquire whether or not $Z$ is also closed,
which is another standard property of zeros of maximally monotone
operators. The next example illustrates that $Z$ may fail to be
closed.

\begin{example}[$Z$ need not be closed!]
Suppose that $X=\ell_2$, the real Hilbert space of
square-summable sequences.
In \cite[Example~3.17]{BWY10},
the authors provide a monotone discontinuous linear at most
single-valued operator $S$ on $X$ such that $S$ is maximally monotone
and its adjoint $S^*$ is a maximally monotone
single-valued extension of $-S$.
Hence $\dom S$ is not closed. 
Now assume that $A=S$ and $B=S^*$.
Then $A+B$ is operator that is zero on the dense proper subspace
$Z=\dom(A+B)=\dom S$ of $X$. Thus $Z$ fails to be closed.
Furthermore, in the language of Passty's parallel sums
(see Remark~\ref{r:parallel}), this also illustrates that
the parallel sum need not map a point to a closed set.
\end{example}

\begin{remark}\label{counterexample} We do not know whether or not
such counterexamples can reside in finite-dimensional Hilbert spaces
when $\dom A \cap \dom B \neq \varnothing$. On the one hand, in
view of the forthcoming 
Corollary~\ref{c:apes}\ref{c:apesi-}, any counterexample
must feature at least one operator that is not paramonotone, which
means that the operators \textit{cannot} be simultaneously
subdifferential operators of functions in $\Gamma$. On the other
hand, one has to avoid the situation when $A+B$ is maximally monotone,
which happens when  $\reli\dom A \cap\reli\dom
B \neq \varnothing$. This means that neither is one of the operators
allowed to have full domain, nor can they simultaneously have
relatively open domains, which excludes the situation when both
operators are maximally monotone linear relations (i.e., maximally
monotone operators with graphs that are linear subspaces, 
see \cite{BWY09}).
\end{remark}

\begin{remark}
We note that $\KK$ and $\ZZ$ are in general \emph{not}
maximally monotone. Indeed if $\ZZ$, say, is maximally monotone,
then Corollary~\ref{c:convex}
and \cite[Theorem~4.2]{BWY09} imply that $\gr\ZZ$ is actually 
\emph{affine} (i.e., a translate of a subspace) and
so are $Z$ and $K$ (as range and domain of $\ZZ$).
However, the set $Z$ of Example~\ref{ex:cf} need not be an
affine subspace (e.g., when $U$, $V$ and $Z$ coincide with
the closed unit ball in $X$).
\end{remark}

\section{Reflected Resolvents and Splitting Operators}

\label{s:splitop} 

We start with some useful identities involving
resolvents and reflected resolvents (recall \eqref{e:alden}). 

\begin{proposition}
\label{p:nocars2}
Let $C\colon X\To X$ be maximally monotone.
Then the following hold.
\begin{enumerate}
\item
\label{p:nocars2i-}
$R_{C^{-1}} = -R_C$.
\item
\label{p:nocars2i}
$J_{C^\ovee} = J_C^\ovee$.
\item
\label{p:nocars2ii}
$R_{C^{-\ovee}} = \Id-2J_C^\ovee$.
\end{enumerate}
\end{proposition}
\begin{proof}
\ref{p:nocars2i-}:
By \eqref{e:iri}, we have
$R_{C^{-1}} = 2J_{C^{-1}}-\Id = 2(\Id - J_C) -\Id = \Id - 2J_C
= -(2J_C-\Id) = -R_C$.

\ref{p:nocars2i}:
Indeed,
\begin{subequations}
\begin{align}
J_{C^\ovee} &= \big(\Id+(-\Id)\circ C\circ(-\Id)\big)^{-1}\\
&=  \big((-\Id)\circ(\Id+C)\circ(-\Id)\big)^{-1}\\
&= (-\Id)^{-1}\circ(\Id+C)^{-1}\circ(-\Id)^{-1}\\
&= (-\Id)\circ J_C \circ (-\Id)\\
&= J_C^\ovee.
\end{align}
\end{subequations}
\ref{p:nocars2ii}:
Using \eqref{e:iri} and \ref{p:nocars2i}, we have that 
$R_{C^{-\ovee}}=2J_{C^{-\ovee}}-\Id =
2(\Id-J_{C^\ovee})-\Id = \Id-2J_C^\ovee$.
\end{proof}

\begin{corollary}[Peaceman-Rachford operator is self-dual]
\label{c:PR}
{\rm (See \textbf{Eckstein}'s \cite[Lemma~3.5 on page~125]{EckThesis}.) }
The Peaceman-Rachford operators for $(A,B)$ and $(A,B)^* =
(A^{-1},B^{-\ovee})$ coincide,
i.e., we have \emph{self-duality}
in the sense that
\begin{equation}
\label{e:PRa}
R_BR_A = R_{B^{-\ovee}}R_{A^{-1}}.
\end{equation}
Consequently,
\begin{equation}
\label{e:PRb}
(\forall \lambda\in[0,1]) \quad
(1-\lambda)\Id + \lambda R_BR_A =
(1-\lambda)\Id + \lambda R_{B^{-\ovee}}R_{A^{-1}}.
\end{equation}
\end{corollary}
\begin{proof}
Using Proposition~\ref{p:nocars2}\ref{p:nocars2i-}\&\ref{p:nocars2ii},
we obtain \eqref{e:PRa}
$R_{B^{-\ovee}}R_{A^{-1}} = (\Id-2J_B^\ovee)(-R_A) =
-R_A+2J_BR_A =(2J_B-\Id)R_A = R_BR_A$.
Now \eqref{e:PRb} follows immediately from \eqref{e:PRa}.
\end{proof}

\begin{corollary}[Douglas-Rachford operator is self-dual]
\label{c:DR}
{\rm (See \textbf{Eckstein}'s \cite[Lemma~3.6 on page~133]{EckThesis}.) }
For the \emph{Douglas-Rachford operator}
\begin{equation}
\label{e:DR}
T_{(A,B)} := \thalb\Id + \thalb  R_BR_A
\end{equation}
we have
\begin{equation}
T_{(A,B)} = J_BR_A+\Id-J_A = T_{(A^{-1},B^{-\ovee})}.
\end{equation}
\end{corollary}
\begin{proof}
The left equality is a simple expansion while
self-duality is \eqref{e:PRb} with $\lambda=\thalb$.
\end{proof}

\begin{remark}[backward-backward operator is not self-dual]
In contrast to Corollary~\ref{c:PR},
the backward-backward operator is not self-dual:
indeed,
using \eqref{e:iri} and Proposition~\ref{p:nocars2}\ref{p:nocars2ii},
we deduce that
\begin{equation}
J_{B^{-\ovee}}J_{A^{-1}} = (\Id-J_B^\ovee)(\Id-J_A)=
\Id-J_A + J_B(J_A-\Id) = (J_B-\Id)(J_A-\Id).
\end{equation}
Thus if $A\equiv 0$ and $\dom B$ is not a singleton 
(equivalently, $J_A =\Id$ and $\ran J_B$ is not a singleton), 
then
$J_{B^{-\ovee}}J_{A^{-1}} \equiv (J_B-\Id)0 \equiv J_B0 \neq J_B=J_BJ_A$.
\end{remark}

For the rest of this paper, we set
\boxedeqn{
\label{e:T}
T = \thalb\Id + \thalb R_BR_A = J_BR_A+\Id-J_A.
}

Clearly,
\begin{equation}
\Fix T = \Fix R_BR_A.
\end{equation}

\begin{theorem}
\label{t:mace2}
The mapping
\begin{equation}
\label{e:mace2}
\Psi\colon\gr\KK\to\Fix T\colon (z,k)\mapsto z+k
\end{equation}
is a well defined bijection that is continuous in both directions,
with $\Psi^{-1}\colon x\mapsto (J_Ax,x-J_Ax)$.
\end{theorem}
\begin{proof}
Take $(z,k)\in\gr\KK$.
Then $k\in K_z = (Az)\cap (-Bz)$.
Now $k\in Az$
$\Leftrightarrow$
$z+k\in(\Id+A)z$
$\Leftrightarrow$
$z = J_A(z+k)$,
and $k\in (-Bz)$
$\Leftrightarrow$
$-k\in Bz$
$\Leftrightarrow$
$z-k\in (\Id+B)z$
$\Leftrightarrow$
$z=J_B(z-k)$.
Set
$x := z+k$.
Then $J_Ax=J_A(z+k)=z$ and
hence $R_Ax = 2J_Ax-x = 2z-(z+k) = z-k$.
Thus,
\begin{equation}
Tx = x-J_Ax+J_BR_Ax =
z+k-z+J_B(z-k) = k+z = x,
\end{equation}
i.e., $x\in\Fix T$.
It follows that $\Psi$ is \emph{well defined}.

Let us now show that $\Psi$ is \emph{surjective}.
To this end, take $x\in \Fix T$. 
Set $z:=J_Ax$ as well as $k:=(\Id-J_A)x = x - z$.
Clearly,
\begin{equation}
x=z+k.
\end{equation}
Now $z=J_Ax$
$\Leftrightarrow$
$x\in(\Id+A)z = z+Az$
$\Leftrightarrow$
$k=x-z\in Az$.
Thus,
\begin{equation}
k\in Az.
\end{equation}
We also have
$R_Ax=2J_Ax-x=2z-(z+k)=z-k$;
hence,
$x=Tx=x-J_Ax+J_BR_Ax$
$\Leftrightarrow$
$J_Ax=J_BR_Ax$
$\Leftrightarrow$
$z=J_B(z-k)$
$\Leftrightarrow$
$z-k\in(\Id+B)z =z+Bz$
$\Leftrightarrow$
\begin{equation}
k\in -Bz.
\end{equation}
Altogether, $k\in (Az)\cap(-Bz) = K_z$
$\Leftrightarrow$
$(z,k)\in\gr \KK$. Hence $\Psi$ is surjective.

In view of Fact~\ref{f:Mintypar} and since $\gr\KK\subseteq\gr A$,
it is clear that $\Psi$ is \emph{injective} with the announced inverse.
\end{proof}

The following result is a straight-forward consequence
of Theorem~\ref{t:mace2}.

\begin{corollary}
\label{c:summerlandrox}
We have
\begin{equation}
\label{e:summerlandrox1}
(\forall z\in Z)
\quad
K_z = J_{A^{-1}}\big(J_A^{-1}z\cap \Fix T\big)
\end{equation}
and
\begin{equation}
\label{e:summerlandrox2}
(\forall k\in K)
\quad
Z_k = J_A(J_{A^{-1}}^{-1}k\cap \Fix T\big).
\end{equation}
\end{corollary}

\begin{corollary}[Combettes]
\label{c:Combettes}
{\rm (see \cite[Lemma~2.6(iii)]{Comb04})}
$J_A(\Fix T) = Z$.
\end{corollary}
\begin{proof}
Set $Q\colon X\times X\to X\colon (x_1,x_2)\mapsto x_1$.
By Theorem~\ref{t:mace2} and Proposition~\ref{p:ZZKK}\ref{p:ZZKKi+},
$J_A(\Fix T) = Q\ran\Psi^{-1} = Q\dom \Psi=Q(\gr\KK) = \dom\KK
=Z$.
\end{proof}

\begin{example}
{\rm (see also \cite[Fact~A1]{BCL02})}
Suppose that $A=N_U$ and $B=N_V$,
where $U$ and $V$ are closed convex subsets of $X$ such that
$U\cap V\neq\varnothing$.
Then $P_U(\Fix T) = U\cap V$.
\end{example}

\begin{corollary}
$(\Id-J_A)(\Fix T) = K$.
\end{corollary}
\begin{proof}
Either argue similarly to the proof
of Corollary~\ref{c:Combettes}, or
apply Corollary~\ref{c:Combettes} to the dual and recall that
$T$ is self-dual by Corollary~\ref{c:DR}.
\end{proof}

\section{Paramonotonicity}

\label{s:para}

\begin{definition}
A monotone operator $C\colon X\To X$ is
\emph{paramonotone},
if
\begin{equation}
\left.
\begin{matrix}
x^*\in Cx\\
y^*\in Cy\\
\scal{x-y}{x^*-y^*}=0
\end{matrix}
\right\}
\quad\Rightarrow\quad
x^*\in Cy \text{~and~} y^*\in Cx.
\end{equation}
\end{definition}

\begin{remark}
Paramonotonicity has proven to be a very useful
property for finding solution of variational
inequalities by iteration;
see, e.g.,
\cite{Iusem98}, \cite{CIZ}, \cite{BurIus98}, \cite{OY03}, and \cite{HS06}.
Examples of paramonotone operators abound:
indeed, each of the following is paramonotone.
\begin{enumerate}
\item $\partial f$, where $f\in\Gamma$
{\rm \cite[Proposition~2.2]{Iusem98}.}
\item $C\colon X\To X$, where $C$ is strictly monotone.
\item $\RR^{n}\to\RR^{n}\colon x\mapsto Cx+b$,
where
$C\in\RR^{n\times n}$, $b\in\RR^n$,
$C_+ = \thalb C+\thalb C^T$,
$\ker C_+\subseteq\ker C$,
and
$C_+$ is positive semidefinite
{\rm \cite[Proposition~3.1]{Iusem98}.}
\end{enumerate}
For further examples, see \cite{Iusem98}.
When $C$ is a continuous linear monotone operator,
then $C$ is paramonotone
if and only if $C$ is rectangular
(a.k.a.\ 3* monotone); see \cite[Section~4]{BBW06}.
It is straight-forward to check that for $C\colon X\To X$, we have
\begin{align}
\label{e:parainv}
\text{$C$ is paramonotone}
&\Leftrightarrow
\text{$C^{-1}$ is paramonotone}\notag\\
&\Leftrightarrow
\text{$C^{\ovee}$ is paramonotone}\\
&\Leftrightarrow\;
\text{$C^{-\ovee}$ is paramonotone.}\notag
\end{align}
\end{remark}

\begin{theorem}
\label{t:apes}
Suppose that $A$ and $B$ are paramonotone.
Then $(\forall z\in Z)$ $K_z=K$ and
$(\forall k\in K)$ $Z_k = Z$.
\end{theorem}
\begin{proof}
Suppose that $z_1$ and $z_2$ belong to $Z$
and that $z_1\neq z_2$.
Take $k_1\in K_{z_1}=Az_1\cap (-Bz_1)$
and $k_2\in K_{z_2} = Az_2\cap (-Bz_2)$.
By Corollary~\ref{c:Passty},
\begin{equation}
\scal{k_1-k_2}{z_1-z_2}=0.
\end{equation}
Since $A$ and $B$ are paramonotone,
we have $k_2\in Az_1$ and $-k_2\in Bz_1$;
equivalently, $k_2\in K_{z_1}$.
It follows that $K_{z_2}\subseteq K_{z_1}$.
Since the reverse inclusion follows in the same fashion,
we see that $K_{z_1}=K_{z_2}$.
In view of Proposition~\ref{p:momday}\ref{p:momdayii},
$K_{z_1}=K$, which proves the first conclusion.
Since $A$ and $B$ are paramonotone so are
$A^{-1}$ and $B^{-\ovee}$ by \eqref{e:parainv}.
Therefore, the second conclusion follows from what we already
proved (applied to $A^{-1}$ and $B^{-\ovee}$).
\end{proof}

\begin{remark}[recovering \emph{all} primal solutions
from \emph{one} dual solution] \
\label{r:apes}
Suppose that $A$ and $B$ are paramonotone
and we know \emph{one} (arbitrary) dual solution, say $k_0\in K$.
Then
\begin{equation}
Z_{k_0} = A^{-1}k_0 \cap \big(B^{-1}(-k_0)\big)
\end{equation}
recovers the set $Z$ of \emph{all} primal solutions,
by Theorem~\ref{t:apes}.
If $A=\partial f$ and $B=\partial g$,
where $f$ and $g$ belong to $\Gamma$,
then, since $(\partial f)^{-1}=\partial f^*$
and $(\partial g)^{-1}=\partial g^*$,
we obtain a formula well known in \emph{Fenchel duality}, namely,
\begin{equation}
Z = \partial f^*(k_0)\cap \partial g^*(-k_0).
\end{equation}
We shall revisit this setting in more detail in Section~\ref{s:Fenchel}.
%see, e.g., \cite[Proposition~XII.5.4.1 on p.~190]{HULL2}.
In striking contrast, the complete recovery
of all primal solutions from one dual solution is generally impossible
when at least one of the operators is no longer
paramonotone --- see, e.g., Example~\ref{ex:normskew} where one
of the operators is even a normal cone operator.

\end{remark}

\begin{corollary}
\label{c:apes}
Suppose $A$ and $B$ are paramonotone.
Then the following hold.
\begin{enumerate}
\item
\label{c:apesi-}
$Z$ and $K$ are closed.
\item
\label{c:apesi}
$\gr\KK$ and $\gr\ZZ$ are the
``rectangles'' $Z\times K$ and $K\times Z$, respectively.
\item
\label{c:apesi+}
$\Fix T = Z+K$.
\item
\label{c:apesii}
$(Z-Z)\perp(K-K)$.
\item
\label{c:apesiii}
$\cspan(K-K)=X$ $\Rightarrow$ $Z$ is a singleton.
\item
\label{c:apesiv}
$\cspan(Z-Z)=X$ $\Rightarrow$ $K$ is a singleton.
\end{enumerate}
\end{corollary}
\begin{proof}
\ref{c:apesi-}:
Combine Theorem~\ref{t:apes} and Proposition~\ref{p:momday}.

\ref{c:apesi}:
Clear from Theorem~\ref{t:apes}.

\ref{c:apesi+}:
Combine \ref{c:apesi} with Theorem~\ref{t:mace2}.

\ref{c:apesii}:
Combine Corollary~\ref{c:Passty} with Theorem~\ref{t:apes}.

\ref{c:apesiii}:
In view of \ref{c:apesii}, we have that 
$0=\scal{Z-Z}{K-K} = \scal{Z-Z}{\cspan(K-K)}
=\scal{Z-Z}{X}$
$\Rightarrow$
$Z-Z=\{0\}$
$\Leftrightarrow$
$Z$ is a singleton.

\ref{c:apesiv}:
This is verified analogously to the proof of \ref{c:apesiii}.
\end{proof}

\begin{corollary}
\label{c:2GB}
Suppose that $A$ and $B$ are paramonotone.
Then $\Fix T = Z+K$,
$Z = J_A(Z+K)$ and
$K = J_{A^{-1}}(Z+K) = (\Id-J_A)(Z+K)$.
\end{corollary}
\begin{proof}
Combine Corollary~\ref{c:apes}\ref{c:apesi}
with Theorem~\ref{t:mace2}.
\end{proof}

\begin{remark}[paramonotonicity is critical]
Various results in this section
--- e.g., Theorem~\ref{t:apes}, Remark~\ref{r:apes},
Corollary~\ref{c:apes}\ref{c:apesi}--\ref{c:apesiv}--- fail
if the assumption of paramonotonicity is omitted.
To generate these counterexamples, assume that
$A$ and $B$ are as in Example~\ref{ex:skewskew}
or Example~\ref{ex:normskew}. 
\end{remark}

\section{Projection operators and solution sets}

\label{s:poss}

The following two facts regarding projection operators will
be used in the sequel.

\begin{fact}
\label{f:osum}
{\rm (See, e.g., \cite[Proposition~2.6]{BCL06}.)}
Let $U$ and $V$ be nonempty closed convex subsets of $X$ such that
$U\perp V$. Then $U+V$ is convex and closed, and $P_{U+V}=P_U + P_V$.
\end{fact}

\begin{fact}
\label{f:translate}
Let $S$ be a nonempty subset of $X$, and let $y\in X$.
Then $(\forall x\in X)$ $P_{y+S}(x) = y+P_S(x-y)$.
\end{fact}

\begin{theorem}
\label{t:oprep}
Suppose that $A$ and $B$ are paramonotone,
that $(z_0,k_0)\in Z\times K$, and that $x\in X$.
Then the following hold.
\begin{enumerate}
\item
\label{t:oprep1}
$Z+K$ is convex and closed.
\item
\label{t:oprep1+}
$P_{Z+K}(x) = P_Z(x-k_0) + P_K(x-z_0)$.
\item
\label{t:oprep2}
If $(Z-Z)\perp K$, then
$P_{Z+K}(x) = P_Z(x) + P_K(x-z_0)$.
\item
\label{t:oprep3}
If $Z\perp(K-K)$, then
$P_{Z+K}(x) = P_Z(x-k_0) + P_K(x)$.
\end{enumerate}
\end{theorem}
\begin{proof}
\ref{t:oprep1}:
The convexity and closedness of $Z$ and $K$ follows
from Corollary~\ref{c:convex} and Corollary~\ref{c:apes}\ref{c:apesi-}.
By Corollary~\ref{c:apes}\ref{c:apesii},
\begin{equation}
(Z-z_0)\perp(K-k_0).
\end{equation}
Using Fact~\ref{f:osum},
\begin{equation}
\label{e:0827a}
Z+K-z_0-k_0\;\text{is convex and closed, and}\;
P_{Z+K-z_0-k_0} = P_{Z-z_0} + P_{K-k_0}.
\end{equation}
Hence $Z+K$ is convex and closed.
\ref{t:oprep1+}:
Using \eqref{e:0827a}, Fact~\ref{f:osum}, and Fact~\ref{f:translate}, 
we obtain
\begin{subequations}
\begin{align}
P_{Z+K}x &= P_{(z_0+k_0)+(Z+K-z_0-k_0)}x\\
&=z_0+k_0 + P_{(Z-z_0)+(K-k_0)}\big(x-(z_0+k_0)\big)\\
&= z_0+P_{Z-z_0}\big((x-k_0)-z_0\big)
+ k_0 + P_{K-k_0}\big((x-z_0)-k_0\big)\\
&=P_Z(x-k_0) + P_K(x-z_0).
\end{align}
\end{subequations}
\ref{t:oprep2}:
Using Fact~\ref{f:osum} and Fact~\ref{f:translate}, we have
\begin{subequations}
\begin{align}
P_{Z+K}x &= P_{z_0+(Z+K-z_0)}x\\
&=z_0 + P_{(Z-z_0)+K}(x-z_0)\\
&= z_0+ P_{Z-z_0}(x-z_0) + P_K(x-z_0)\\
&= P_Zx+P_K(x-z_0).
\end{align}
\end{subequations}
\ref{t:oprep3}: Argue analogously to the proof of \ref{t:oprep2}.
\end{proof}

\begin{remark}
Suppose that $A$ and $B$ are paramonotone and that $0\in K$.
Then Corollary~\ref{c:apes}\ref{c:apesii} implies that
$(Z-Z)\perp K-\{0\} = K$ and we thus may employ
either item~\ref{t:oprep1+} (with $k_0=0$) or item~\ref{t:oprep2} to
obtain the formula for $P_{Z+K}$.
\end{remark}

However, if $(Z-Z)\perp K$, then the next two examples show---in strikingly
different ways since $Z$ is either large or small---that
we cannot conclude that $0\in K$:

\begin{example}
Fix $u\in X$ and suppose that $(\forall x\in X)$ $Ax=u$
and $B=-A$. Then $A$ and $B$ are paramonotone, $A+B\equiv 0$,
and hence $Z=X$. Furthermore, $K=\{u\}$.
Thus if $u\neq 0$, then $K\not\perp X = (Z-Z)$.
\end{example}

\begin{example}
Let $U$ and $V$ be closed convex subsets of $X$
such that
\begin{equation}
0\notin U\cap V
\text{~and~}
U-V=X.
\end{equation}
(For example, suppose that $X=\RR$ and set $U=V=\left[1,+\infty\right[$.)
Now assume that $(A,B) = (N_U,N_V)^*$.
In view of Example~\ref{ex:cf},
$K = U\cap V$ and $Z = N_{\overline{U-V}}(0) = N_X(0)=\{0\}$.
Hence $Z$ is a singleton and thus $Z-Z=\{0\}\perp K$ while $0\notin K$.
\end{example}

\begin{theorem}
\label{t:lousypizza}
Suppose that $A$ and $B$ are paramonotone,
let $k_0\in K$, and let $x\in X$.
Then the following hold.
\begin{enumerate}
\item
\label{t:lousypizza1}
$J_AP_{Z+K}(x) = P_Z(x-k_0)$.
\item
\label{t:lousypizza2}
If $(Z-Z)\perp K$, then
$J_A\circ P_{Z+K}= P_Z$.
\end{enumerate}
\end{theorem}
\begin{proof}
Take an arbitrary $z_0\in Z$. 
\ref{t:lousypizza1}:
Set $z := P_Z(x-k_0)$.
Using Theorem~\ref{t:oprep}\ref{t:oprep1+}
and Theorem~\ref{t:apes}, we have
\begin{equation}
P_{Z+K}x - z =
P_{Z+K}x - P_Z(x-k_0) = P_K(x-z_0) \in K = K_z
\subseteq Az.
\end{equation}
Hence $P_{Z+K}x \in (\Id+A)z$
$\Leftrightarrow$
$z = J_AP_{Z+K}x$
$\Leftrightarrow$
$P_Z(x-k_0) = J_AP_{Z+K}x$.

\ref{t:lousypizza2}:
This time, let us set $z := P_Zx$.
Using Theorem~\ref{t:oprep}\ref{t:oprep2}
and Theorem~\ref{t:apes}, we have
\begin{equation}
P_{Z+K}x - z =
P_{Z+K}x - P_Zx = P_K(x-z_0) \in K = K_z
\subseteq Az.
\end{equation}
Hence 
$P_{Z+K}x \in (\Id+A)z$
$\Leftrightarrow$
$z = J_AP_{Z+K}x$
$\Leftrightarrow$
$P_Zx = J_AP_{Z+K}x$.
\end{proof}

\begin{corollary}
\label{c:lousypizza}
Suppose that $A$ and $B$ are paramonotone,
and that $0\in K$.
Then
\begin{equation}
P_Z = J_AP_{Z+K}.
\end{equation}
%$(\forall x\in X)$ $P_Z(x) = J_AP_{Z+K}x$.
\end{corollary}

Specializing the previous result to normal cone operators,
we recover the consistent case of
\cite[Corollary~3.9]{BCL04}.

\begin{example}
\label{ex:summerland}
Suppose that $A=N_U$ and $B=N_V$,
where $U$ and $V$ are closed convex subsets of $X$
such that $U\cap V\neq\varnothing$.
Then $Z=U\cap V$, $K= N_{\overline{U-V}}(0)$, and
\begin{equation}
P_{Z} = P_UP_{Z+K} = P_UP_{\Fix T}.
\end{equation}
\end{example}
\begin{proof}
This follows from Example~\ref{ex:cf},
Corollary~\ref{c:apes}\ref{c:apesi+},
and Corollary~\ref{c:lousypizza}.
\end{proof}

\section{Subdifferential operators}

\label{s:Fenchel}

In this section, we assume that
\begin{equation}
A = \partial f
\text{~and~}
B = \partial g,
\end{equation}
where $f$ and $g$ belong to $\Gamma$.
We consider the \emph{primal problem}
\begin{equation}
\label{e:Fprimal}
\minimize{x\in X}{f(x)+g(x)}
\end{equation}
the associated \emph{Fenchel dual problem}
\begin{equation}
\label{e:Fdual}
\minimize{x^*\in X}{f^*(x^*)+g^*(-x^*)},
\end{equation}
the \emph{primal} and \emph{dual optimal values}
\begin{equation}
\mu = \inf(f+g)(X)
\text{~and~}
\mu^* = \inf(f^*+g^{*\veet})(X).
\end{equation}
Note that
\begin{equation}
\mu \geq -\mu^*.
\end{equation}
Following  \cite{BGWa} and \cite{BGWb}, we say that
\emph{total duality} holds if
$\mu=-\mu^*\in\RR$, the primal problem \eqref{e:Fprimal}
has a solution, and the dual problem \eqref{e:Fdual}
has a solution.

\begin{theorem}[total duality]
Suppose that $A=\partial f$ and $B=\partial g$,
where $f$ and $g$ belong to $\Gamma$.
Then
$Z\neq\varnothing$
$\Leftrightarrow$
\emph{total duality} holds,
in which case
$Z$ coincides with the set of solutions to the primal problem
\eqref{e:Fprimal}.
\end{theorem}
\begin{proof}
Observe that $(\partial f)^{-1}= \partial f^*$ and
that $(\partial g)^{-\ovee} = (\partial g^*)^\ovee
= \partial(g^{*\veet})$.

``$\Rightarrow$'':
Suppose that $Z\neq\varnothing$, and let $z\in Z$.
Then $0\in\partial f(z) + \partial g(z) \subseteq
\partial(f+g)(z)$.
Hence $z$ solves the primal problem \eqref{e:Fprimal},
and
\begin{equation}
\mu = f(z) + g(z).
\end{equation}
Take $k\in K=K_z = (\partial f)(z) \cap (-\partial g)(z)$.
First, we note that
$0\in (\partial f)^{-1}(k) + (\partial g)^{-\ovee}(k)
= \partial f^*(k) + \partial g^{*\veet}(k)
\subseteq \partial (f^*+g^{*\veet})(k)$ and so
$k$ solves the Fenchel dual problem \eqref{e:Fdual}.
Thus,
\begin{equation}
\mu^* = f^*(k) + g^{*\veet}(k).
\end{equation}
Moreover, $k\in \partial f(z)$ and $-k\in\partial g(z)$, i.e.,
$f(z)+f^*(k)=\scal{z}{k}$ and $g(z) + g^*(-k)=\scal{z}{-k}$.
Adding these equations gives
$0 = f(z)+f^*(k) + g(z) + g^{*\veet}(k) = \mu+\mu^*$.
This verifies total duality.

``$\Leftarrow$'':
Suppose we have total duality. Then there exists
$x\in \dom f\cap \dom g$ and $x^*\in \dom f^*\cap \dom g^{*\veet}$
such that
\begin{equation}
\label{e:lastbeach3}
f(x) + g(x) = \mu = -\mu^* = -f^*(x^*) - g^{*\veet}(x^*)\in\RR.
\end{equation}
Hence $0 = (f(x)+f^*(x^*)) + (g(x)+g^*(-x^*))
\geq \scal{x}{x^*} + \scal{x}{-x^*}=0$.
Therefore, using convex analysis and Proposition~\ref{p:momday},
\begin{equation}
\big(x^*\in\partial f(x) \text{~and~} -x^*\in\partial g(x)\big)
\;\Leftrightarrow\;
x^*\in K_x
\;\Leftrightarrow\;
x \in Z_{x^*}.
\end{equation}
Hence $x\in Z$.

Note that
$Z=\zer(\partial f + \partial g) \subseteq
\zer\partial (f+g)$ since
$\gr(\partial f + \partial g)\subseteq \gr\partial(f+g)$.
Hence $Z$ is a subset of the set of primal solutions.
Conversely, if $x$ is a primal solution and $x^*$ is a dual solution,
then
\eqref{e:lastbeach3} holds and the rest of the proof
of ``$\Leftarrow$'' shows that $x\in Z$.
Altogether, $Z$ coincides with the set of primal solutions.
\end{proof}

\begin{remark}[sufficient conditions]
On the one hand,
\begin{equation}
\label{e:lastbeach1}
\text{
the primal problem has at least one solution}
\end{equation} if
$\dom f\cap\dom g\neq\emp$ and one of the following holds
(see, e.g., \cite[Corollary~11.15]{BC2011}):
(i) $f$ is supercoercive;
(ii) $f$ is coercive and $g$ is bounded below;
(iii) $0 \in \sri(\dom f^* + \dom g^*)$
(by, e.g., \cite[Proposition~15.13]{BC2011} and
since \eqref{e:Fprimal} is the 
\emph{Fenchel dual problem} of \eqref{e:Fdual}). 
On the other hand,
\begin{equation}
\label{e:lastbeach2}
\text{the sum rule $\partial(f+g) = \partial f + \partial
g$ holds}
\end{equation}
whenever one of the following is satisfied
(see \cite[Corollary~16.38]{BC2011}):
(i) $0\in\sri(\dom f - \dom g)$;
(ii) $\dom f \cap \inte\dom g\neq\varnothing$;
(iii) $\dom g = X$;
(iv) $X$ is finite-dimensional and $\reli\dom f \cap \reli\dom
g\neq\varnothing$.
If both \eqref{e:lastbeach1} and \eqref{e:lastbeach2} hold, then
$Z\neq\varnothing$ and  $Z$ coincides with the set of primal
solutions.
\end{remark}

\section{Algorithms and Eckstein-Ferris-Robinson Duality}

\label{s:final}

In this last section, we sketch first algorithmic consequences
and then conclude by commenting 
on the applicability of our work to a more general duality
framework.

\begin{theorem}[abstract algorithm]
\label{t:abstract}
Suppose that $A$ and $B$ are paramonotone.
Let $(x_n)_\nnn$ be a sequence
such that
$(x_n)_\nnn$ converges (weakly or strongly) to $x\in \Fix T$
and $(J_Ax_n)_\nnn$ converges (weakly or in norm) to $J_Ax$.
Then the following hold.
\begin{enumerate}
\item
$(\forall k\in K)$ $J_Ax=P_Z(x-k)$.
\item
If $(Z-Z)\perp K$, then $J_Ax = P_Zx$.
\end{enumerate}
\end{theorem}
\begin{proof}
Combine Corollary~\ref{c:apes}\ref{c:apesi+}
with Theorem~\ref{t:lousypizza}.
\end{proof}

We provide three examples.

\begin{example}[Douglas-Rachford algorithm]
Suppose that $A$ and $B$ are paramonotone and that 
the sequence $(x_n)_\nnn$ is generated by
$(\forall\nnn)$ $x_{n+1} = Tx_n$.
The hypothesis in Theorem~\ref{t:abstract}
is satisfied, and the convergence of the
sequences is with respect to the \emph{weak topology}
\cite{Svaiter}.
See also \cite{B2011} for a much simpler proof and
\cite[Theorem~25.6]{BC2011} for a powerful generalization.
\end{example}

\begin{example}[Halpern-type algorithm]
\label{ex:Halpern}
Suppose that $A$ and $B$ are paramonotone and that 
the sequence $(x_n)_\nnn$ is generated by
$(\forall\nnn)$
$x_{n+1} = (1-\lambda_n)Tx_n + \lambda_n y$,
where $(\lambda_n)_\nnn$ is a
sequence of parameters in $\zeroun$ and $y\in X$ is given.
Under suitable assumptions on $(\lambda_n)_\nnn$,
it is known (see, e.g., \cite{Halpern}, \cite{Wittmann})
that
$x_n\to x := P_{\Fix T}y$ with respect to the \emph{norm topology}.
Since $J_A$ is (firmly) nonexpansive, it is clear
that the hypothesis of Theorem~\ref{t:abstract} holds.
Furthermore, $J_Ax_n\to J_Ax = J_AP_{\Fix T}y$.
Thus, if $k_0\in K$, then
$J_Ax_n\to P_Z(y-k_0)$ by Theorem~\ref{t:lousypizza}\ref{t:lousypizza1}.
And if $(Z-Z)\perp K$, then $J_Ax_n\to P_Zy$
by Theorem~\ref{t:lousypizza}\ref{t:lousypizza2}.
\end{example}

\begin{example}[Haugazeau-type algorithm]
This is similar to Example~\ref{ex:Halpern}
in that $x_n\to x:= P_{\Fix T}y$ with respect
to the \emph{norm topology} and where $y\in X$ is given.
For the precise description of the (somewhat complicated)
update formula for $(x_n)_\nnn$,
we refer the reader to \cite[Section~29.2]{BC2011} or
\cite{BC01}; see also \cite{Haugazeau}.
Once again, we have  $J_Ax_n\to J_Ax = J_AP_{\Fix T}y$ and
thus, if $k_0\in K$, then
$J_Ax_n\to P_Z(y-k_0)$ by Theorem~\ref{t:lousypizza}\ref{t:lousypizza1}.
And if $(Z-Z)\perp K$, then $J_Ax_n\to P_Zy$
by Theorem~\ref{t:lousypizza}\ref{t:lousypizza2}.
Consequently, in the context of Example~\ref{ex:summerland},
we obtain $P_Ux_n \to P_{U\cap V}y$; in fact,
this is \cite[Theorem~3.3]{BCL06}, which is the main result of
\cite{BCL06}.
\end{example}

Turning to Eckstein-Ferris-Robinson duality, 
let us assume the following:
\begin{itemize}
\item 
$Y$ is a real Hilbert space (and possibly different from $X$);
\item
$C$ is a maximally monotone operator on $Y$; 
\item
$L\colon X\to Y$ is continuous and linear. 
\end{itemize}
Eckstein and Ferris \cite{EckFer}
as well as Robinson \cite{Robinson99} consider the problem of finding
zeros of 
\begin{equation}
A+L^*CL. 
\end{equation}
This framework is more flexible than the Attouch-Th\'era framework,
which corresponds to the case when $Y=X$ and $L=\Id$. 
Note that just as Attouch-Th\'era duality relates to 
classical Fenchel duality in the subdifferential case
(see Section~\ref{s:Fenchel}),
the Eckstein-Ferris-Robinson duality pertains to classical
\emph{Fenchel-Rockafellar duality} for the problem of minimizing
$f+h\circ L$ when 
$f\in\Gamma_X$ and $h\in\Gamma_Y$,  
and $A=\partial f$ and $C=\partial h$.

The results in the previous sections can be used in
the Eckstein-Ferris-Robinson framework thanks to 
items \ref{p:gettingthere2} and \ref{p:gettingthere3} of 
the following result, which allows us to set $B = L^*CL$. 

\begin{proposition}
\label{p:gettingthere}
The following hold.
\begin{enumerate}
\item 
\label{p:gettingthere1}
If $C$ is paramonotone, then $L^*CL$ is paramonotone.
\item 
\label{p:gettingthere2}
{\rm \textbf{(Pennanen)}}
If $\RR_{++}(\ran L-\dom C)$ is a closed subspace of $Y$,
then $L^*CL$ is maximally monotone. 
\item 
\label{p:gettingthere3}
If $C$ is paramonotone and  $\RR_{++}(\ran L-\dom C)$ is a closed
subspace of $Y$, then $L^*CL$ is maximally monotone and paramonotone.
\end{enumerate}
\end{proposition}
\begin{proof}
\ref{p:gettingthere1}:
Take $x_1$ and $x_2$ in $X$,
and suppose that $x_1^*\in L^*CLx_1$ and $x_2^*\in L^*CLx_2$. 
Then there exist $y_1^*\in CLx_1$ and $y_2^*\in CLx_2$
such that $x_1^*=L^*y_1^*$ and $x_2^*=L^*y_2^*$. 
Thus, 
$\scal{x_1-x_2}{x_1^*-x_2^*}=\scal{x_1-x_2}{L^*y_1^*-L^*y_2^*}
=\scal{Lx_1-Lx_2}{y_1^*-y_2^*}\geq 0$ because $C$ is monotone. 
Hence $L^*CL$ is monotone. 
Now suppose furthermore that $\scal{x_1-x_2}{x_1^*-x_2^*}=0$.
Then $\scal{Lx_1-Lx_2}{y_1^*-y_2^*}=0$
and the paramonotonicity of $C$ yields
$y_2^*\in C(Lx_1)$ and $y_1^*\in C(Lx_2)$. 
Therefore,
$x_2^*=L^*y_2^*\in L^*CLx_1$ and 
$x_1^*=L^*y_1^*\in L^*CLx_2$. 

\ref{p:gettingthere2}:
See \cite[Corollary~4.4.(c)]{Pennanen00}. 

\ref{p:gettingthere3}:
Combine \ref{p:gettingthere1} and \ref{p:gettingthere2}. 
\end{proof}

\section*{Acknowledgments}
Part of this research benefited from discussions
during a research visit of HHB
related to a study leave at the Technical University of Chemnitz.
HHB thanks Dr.~Gert Wanka, his optimization group,
and the Faculty of Mathematics for their hospitality.
HHB was partially supported by the Natural Sciences and
Engineering Research Council of Canada and by the Canada Research Chair
Program.
RIB was partially supported by the German Research Foundation.
WLH was partially
supported by the Natural Sciences and Engineering Research Council
of Canada.
WMM was partially supported
by a University Graduate Fellowship of UBC.

%\small

\end{document}